\documentclass{amsart}
\usepackage{amsmath,amsthm,amssymb,latexsym,ifthen,graphicx,comment,tikz,enumerate}
\usepackage{mathrsfs}
\usepackage{color}
\usepackage[colorlinks,citecolor=blue,urlcolor=black, linkcolor=blue, backref]{hyperref}

\usepackage{mathrsfs}
\newtheorem{theorem}{Theorem}[section]

\newtheorem{corollary}[theorem]{Corollary}

\newtheorem{definition}[theorem]{Definition}

\newtheorem{First Proof}[theorem]{First Proof}
\newtheorem{fact}[theorem]{Fact}
\newtheorem{Second Proof}[theorem]{Second Proof}

\newtheorem{lemma}[theorem]{Lemma}

\newtheorem{proposition}[theorem]{Proposition}
\newtheorem{question}{Question}
\newtheorem{remark}[theorem]{Remark}

\begin{document}


\title{A Variant of Chaitin's Omega Function}

\author{Yuxuan Li}
\address{School of Philosophy\\
Fudan University\\
Shanghai 200433,
People's Republic of China}
\email{leecosin16@gmail.com}
\author{Shuheng Zhang}
\address{Graduate School of Informatics\\
Nagoya University\\
Japan}
\email{tomiaoto@gmail.com}
\author{Xiaoyan Zhang}
\address{Department of Mathematics\\
National University of Singapore\\
Singapore 119076}
\email{e0678309@u.nus.edu}
\author{Xuanheng Zhao}
\address{School of Mathematics\\
  Nanjing University\\
  Nanjing, Jiangsu 210093, People's Republic of China}
\email{xuanheng21@gmail.com}


\begin{abstract}  We investigate the continuous function $f$ defined by $$x\mapsto \sum_{\sigma\le_L x    }2^{-K(\sigma)}$$ as a variant of Chaitin's Omega from the perspective of analysis, computability, and algorithmic randomness. Among other results, we obtain that: (i) $f$ is differentiable precisely at density random points;  (ii) $f(x)$ is $x$-random if and only if $x$ is weakly low for $K$ (low for $\Omega$); (iii) the range of $f$ is a null, nowhere dense, perfect $\Pi^0_1(\emptyset')$ class with Hausdorff dimension $1$; (iv)
$f(x)\oplus x\ge_T\emptyset'$ for all $x$; (v) there are $2^{\aleph_0}$ many $x$ such that $f(x)$ is not 1-random; (vi) $f$ is not Turing invariant but is Turing invariant on the ideal of $K$-trivial reals. We also discuss the connection between $f$ and other variants of Omega.
\end{abstract}
\maketitle

\section{Introduction}
In this paper we fix an optimal prefix-free machine $U$ and we usually use $K(\sigma)$ to denote $K_U(\sigma)$.

The real $\Omega=\Omega_U=\sum_{U(\sigma)\downarrow} 2^{-|\sigma|}$ was defined by Chaitin \cite{chaitin_1975}, and is now called Chaitin's Omega. An overview of the study of Omega can be found in Barmpalias \cite{MR4382443}. Omega is the halting probability of an optimal prefix-free machine, which is a natural example for a left-c.e.\ and 1-random real. Later Kučera and Slaman \cite{kucera} showed that all the 1-random left-c.e.\ reals are the halting probabilities of universal prefix-free machines, which makes Omega interesting and significant in the theory of algorithmic randomness. Since then, the Omega number has been generalized in many different ways: (i) Fix a universal prefix-free oracle machine $V$, Downey et al.\  \cite{downey2005relativizing} defined and studied the Omega operator
$x\mapsto \sum_{V^x(\sigma)\downarrow} 2^{-|\sigma|}$. This operator is not continuous and not Turing invariant. (ii) H\"olzl et al.\ \cite{hölzl_merkle_miller_stephan_yu_2020} introduced a way to view Chaitin's Omega as a continuous function on the reals. They defined the map
$x\mapsto \sum_{\sigma\prec x}2^{-K_U(\sigma)}$
and showed this function is differentiable exactly at the 1-random reals, along with some algorithmic randomness properties, such as the maximum value must be random, the image of $x$ is $x$-random if and only if $x$ is weakly low for $K$ along itself. (iii) Also see Section 6 for the variant of Omega in Becher et al. \cite{MR2275867}.

In this paper, we investigate the properties of another variant of Chaitin's Omega as a continuous function, which was introduced in the second author's unpublished master thesis: let $f$ be the function from the Cantor space $2^\omega$ to $[0,1]$ defined by 
$$f:x\mapsto \sum_{\sigma\le_L x}2^{-K_U(\sigma)},$$
where $\sigma <_L x$ means $\sigma$ is on the left of $x$ and $\sigma\le_L x$ means $\sigma <_L x$ or $\sigma$ is an initial segment of $x$. Formally, $\sigma <_L x$ means that there is $n$ such that $\sigma \upharpoonright n=x \upharpoonright n$, $\sigma (n)=0$ and $x(n)=1$. In order to prevent our function from being confused with other variants of Chaitin's Omega, we use $f$ instead of the traditional symbol $\Omega$.

The outline of this paper is as follows. In Section 2, we provide the preliminaries on computability, algorithmic randomness and Hausdorff dimension.

In Section 3 we discuss the function $f$ from the perspective of mathematical analysis. Given any real $x$, let $\hat{\Omega}(x)=\sum_{\sigma\prec x}2^{-K_U(\sigma)}$, which was defined in \cite{hölzl_merkle_miller_stephan_yu_2020} as discussed above. H\"olzl et al.\ \cite{hölzl_merkle_miller_stephan_yu_2020} showed that $\hat{\Omega}$ is differentiable exactly at 1-random reals and the derivative is always 0 provided it exists. We prove that our function $f$ is differentiable exactly at density random points (see Definition \ref{density}). Note that density randomness implies 1-randomness and the former is a strictly stronger notion since Day and Miller
\cite{MR3350101} showed that for a 1-random real, to be density random is a stronger randomness condition than to be a positive density point. Also the derivative of $f$ is always 0 provided it exists. While $\hat{\Omega}$ is nowhere monotone (Theorem 4.6 in \cite{hölzl_merkle_miller_stephan_yu_2020}), $f$ is strictly monotone increasing in the sense that $x<_Ly$ implies $f(x)<f(y)$ by its definition. From this perspective, the structure of the range of $f$ seems much easier to seize than $\hat{\Omega}(2^\omega)$. In fact, we show that $f(2^\omega)$ is a null (of Lebesgue measure zero), nowhere dense (its closure has empty interior) and perfect (it contains no isolated point) set. This sounds like a Cantor ternary set. In fact, $f(2^\omega)$ can be constructed in a similar way as the Cantor set. We will use this basic but important fact tons of times in this paper. For example, the Hausdorff dimension of $f(2^\omega)$ is 1 since it ``contains'' a generalized Cantor set $C_\gamma$ for every scale $\gamma>1$ (see Subsection 2.2 and Theorem \ref{hausdorff} for more details).

In Section 4, we investigate the complexity properties of $f$. $f(2^\omega)$ is a $\Pi^{0}_1(\emptyset')$ class hence no real in $f(2^\omega)$ is $\emptyset'$-Kurtz random. In \cite{hölzl_merkle_miller_stephan_yu_2020} it was shown that $\hat{\Omega}(x) \oplus x \ge_T \emptyset'$ for every $x$. This property also holds for $f$, that is $f(x)\oplus x\ge_T\emptyset'$ for all $x$. Then by the definition of $f$, we conclude that for every real $x$, the join of any two of $x$, $f(x)$ and $\emptyset'$ can compute the other real.

In Section 5, we investigate the randomness properties of $f$. We show that $f(x)$ is $x$-random if and only if $x$ is weakly low for $K$ (which is equivalent to being low for $\Omega$ by Miller \cite{miller2009k}) while $\hat{\Omega}(x)$ is $x$-random if and only if $x$ is weakly low for $K$ along itself in \cite{hölzl_merkle_miller_stephan_yu_2020}. Many corollaries follows from this equivalent characterization of weakly lowness for $K$. Among others, we show that: (i) $f$ is Turing invariant in the ideal induced by the $K$-trivial reals; (ii) (with Slaman) $f$ is not Turing invariant in general; and (iii) the set $\{ x: f(x) \text{ is not 1-random} \}$ is null. By (iii), the reals $x$ such that $f(x)$ is not 1-random seems to be rare. However, we then show with Yu that there indeed exists a real $x \le_T \emptyset'$ such that $f(x)$ is not 1-random. The proof uses an argument to make $f(x)$ not normal and hence not random. We summarize the core steps in the proof as a lemma called the Small Perturbation Lemma. Using this lemma we can also show that there are $2^{\aleph_0}$ many $x$ such that $f(x)$ is not 1-random. 

At the end of Section 5 we introduce two equivalent characterizations of 2-randomness and $\emptyset'$-Kurtz randomness which is related to the structure of the range of $f$. We use tests of the form $\{V_n=\bigcup_{i}[a_i^n,b_i^n]\}_{n\in\omega}$ where $a_i^n,b_i^n$ are uniformly left-c.e.\ reals.

In Section 6, we consider the relation between $f$ and the other variants of Chaitin's Omega including the ones in  \cite{MR2275867} and \cite{downey2005relativizing}. In \cite{MR2275867}, the authors defined $\Omega[X]=\sum_{U(\sigma)\in X} 2^{-|\sigma|}$ for $X \subseteq2^{<\omega}$. $f$ can be viewed as the restriction of $\Omega[\cdot]$ on the sets of the form $\{\sigma:\sigma\le_Lx\}$ for $x\in 2^\omega$ by Fact \ref{502}. Then the positive results in \cite{MR2275867} can be used to derive positive results for our function $f$. For example, it was shown in \cite{MR2275867} that if $n \ge 2$ and $X$ is $\Sigma_n^0$-complete or $\Pi_n^0$-complete, then $\Omega[X]$ is 1-random. An immediate corollary is that if $x$ is left-$\Sigma_n^0$-complete for $n \ge 2$, then $f(x)$ and $\sum_{\sigma<_Lx}2^{-K(\sigma)}$ are 1-random. Becher et al.\ \cite{MR2275867} showed that for a nonempty c.e.\ set $X \subseteq 2^{<\omega}$ and a left-c.e.\ and 1-random real $a\in [0,1]$, there is an optimal machine $V$ such that $a=\Omega_V[X]$. We generalize this result by showing that if $a$ is a left-c.e. and 1-random real, and $x$ is a left-c.e.\ real, then there is an optimal machine $V$ such that $0.a=f_V(x)$. We end Section 6 and the whole paper by raising some questions related to the previous sections.

\section{Preliminaries}
\subsection{Computability theory and algorithmic randomness}
We assume the reader has known basic computability theory and algorithmic randomness, as provided by Soare \cite{MR882921}, Downey-Hirschfeldt \cite{downey2010algorithmic} and Nies \cite{MR2548883}.

We work in the Cantor space $2^\omega$, and use a \textit{real} to denote an element of $2^\omega$ or just an element of the interval $[0,1]$. For finite sequences $\sigma,\tau\in 2^{<\omega}$ or a real $x\in 2^\omega$, we use $\sigma\prec\tau $ or $\sigma\prec x$ to denote that $\sigma$ is an initial segment of $\tau$ or $x$. Let $[\sigma]$ be the class $\{x\in2^\omega :x\succ\sigma   \}$. We use $K(\sigma)$ to denote the prefix-free complexity $K_U(\sigma)$ when there is no ambiguity. We use $\lambda$ to denote the empty string. A real $x$ in $2^\omega$ is \textit{left-c.e.}\ if there is a computable nondecreasing sequence $\{x_s\}$ of rationals such that $\lim_sx_s=x$. Similarly we define the \textit{right-c.e.}\ reals. A real is \textit{computable} if it is left-c.e.\ and right-c.e.. A real is \textit{difference left-c.e.}\ if it is the difference of two left-c.e.\ reals. For reals $x$ and $y$, let $x\oplus y$ be the unique real $z$ such that $z(2i)=x(i)$ and $z(2i+1)=y(i)$ for all $i$. For strings $\sigma$ and $\tau$ such that $|\sigma|=|\tau|$ or $|\sigma|=|\tau|+1$, let $\sigma\oplus\tau$ be the unique string $\rho$ of length $|\sigma|+|\tau|$ such that $\rho(2i)=\sigma(i)$ for all $i<|\sigma|$ and that $\rho(2j+1)=\tau(j)$ for all $j<|\tau|$.

A partial computable function mapping strings to strings is called a \textit{machine}. A machine $M$ is called \textit{prefix-free} if the domain of $M$ is a prefix-free set. A prefix-free machine $U$ is \textit{optimal} if $K_U(\sigma)\le ^+K_M(\sigma)$ for each prefix-free machine $M$, by $f\le^+g$ for two functions  $f,g:\omega\to\omega$ we mean $(\exists c\in \omega)(\forall n)[f(n) \le g(n)+c]$. A c.e.\ set $W\subseteq \omega \times 2^{<\omega}$ is a \textit{Kraft-Chaitin (KC) set} if 
$$\sum_{(n,\sigma)\in W}2^{-n}\le 1.$$

The following theorem reduces the problem of constructing a prefix-free machine to defining a KC set, which is easier to handle.

\begin{theorem}[Machine Existence Theorem, see \cite{MR2548883} P88] For each KC set, one can effectively obtain a prefix-free machine $M$ such that 
$$(\forall r,\sigma)[(r,\sigma)\in W \iff(\exists\tau)[|\tau|=r \And M(\tau)=\sigma]].$$    
\end{theorem}

A real $x$ is \textit{1-random} if $K(x\upharpoonright n) \ge^+n$. For an oracle $z$, a real $x$ is \textit{$z$-random} if $K^z(x\upharpoonright n) \ge^+n$, and is \textit{$n$-random} if it is $\emptyset^{(n-1)}$-random. A real $x$ is \textit{Kurtz random} (also known as weakly random) if it is not in any $\Pi^0_1$ null class. A real $x$ is \textit{weakly 2-random} if it is not in any $\Pi_2^0$ null class.

A \textit{supermatringale} is a function $B:2^{<\omega}\to[0,+\infty)$ that satisfies for every $\sigma$, $B(\sigma0)+B(\sigma1)\le 2B(\sigma)$. A supermartingale $B$ is \textit{computably enumerable} if $B(\sigma)$ is left-c.e.\ uniformly in $\sigma$. Schnorr showed that a real $x$ is 1-random if and only if there is no c.e.\ supermartingale $B$ such that $\sup_nB(x\upharpoonright n)= \infty$.

A real $x$ is \textit{$K$-trivial} if $K(x \upharpoonright n)\le ^+K(n) $. The following are some lowness properties weaker than $K$-triviality: A real $x$ is \begin{enumerate}[(i)]
    \item \textit{weakly low for $K$} if $(\exists c\in \omega)(\exists^\infty n) [K(n)\le K^x(n)+c]$; 
    \item \textit{weakly low for $K$ along itself} if $(\exists c\in \omega)(\exists^\infty n) [K(x\upharpoonright n)\le K^x(x\upharpoonright n)+c]$; 
    \item \textit{low for $\Omega$} if $\Omega$ is $x$-random. 
\end{enumerate} Miller showed that (i) and (iii) are equivalent.

\begin{theorem}[\cite{miller2009k}]\label{Miller}
A real is low for $\Omega$ if and only if it is weakly low for $K$.
\end{theorem}
In particular, every 2-random real is weakly low for $K$ since they are low for $\Omega$. We will also use the following equivalent characterization of $K$-triviality.
\begin{theorem}[Miller, see \cite{downey2010algorithmic} Theorem 15.9.2]\label{dce}
    A real $x$ is $K$-trivial if and only if $x$ is low for difference left-c.e.\ reals (the class of reals that are difference left-$x$-c.e.\ coincides with the class of difference left-c.e.\ reals).
\end{theorem}

For left-c.e.\ reals $x$ and $y$, $x$ is \textit{Solovay reducible} to $y$, written $x\leq_S y$, if there is $d\in\omega$ such that $y-2^{-d}x$ is left-c.e.. A left-c.e.\ real $x$ is \textit{Solovay complete} if $y\leq_S x$ for any left-c.e.\ real $y$. As stated in the introduction section, Kučera and Slaman \cite{kucera} showed that all the left-c.e.\ and 1-random reals $x$ are \textit{Omega numbers}, i.e. there is an optimal prefix-free machine $V$ such that $x=\sum_{V(\sigma)\downarrow} 2^{-|\sigma|}$. Moreover, they are all Solovay-complete.
\begin{theorem}[\cite{kucera}]\label{KS} Suppose that $x \in 2^\omega$ is a 1-random and left-c.e.\ real. Then $x$ is Solovay complete.
\end{theorem}

Then a real is right-c.e.\ and 1-random if and only if it is the complement of an Omega number. Moreover, one does not obtain anything new by considering difference left-c.e.\ reals.
\begin{theorem}[Rettinger and Zheng \cite{rettinger2005solovay}]\label{cecoce}
    A 1-random and difference left-c.e.\ real which is not left-c.e.\ must be a right-c.e.\ real.
\end{theorem}

Given a real $x$, a function $g:\omega \to \omega$ is an \textit{$x$-Solovay function} if $K^x(n)\le^+g(n)$ for all $n$ and $K^x(n)=^+g(n)$ for infinitely many $n$. Given a real $x$, a function $h:\omega \to \omega$ is an \textit{$x$-information content measure} in case $h$ is right-$x$-c.e.\ and $\sum_n2^{-h(n)}$ converges. We use the Solovay functions in the following way. 

\begin{theorem}[Bienvenu and Downey \cite{MR2870648}; H\"olzl, Kr\"aling and Merkle \cite{hölzl_kraling_merkle_2009}] \label{Solo}
   Let $h$ be a $x$-information content measure. Then $h$ is an $x$-Solovay function if and only if $\sum_n2^{-h(n)}$ is $x$-random.
\end{theorem}

\subsection{Density randomness}

The Lebesgue’s density theorem asserts that for almost every point $z$ in a measurable set $C\subseteq [0,1]$, the class is “thick” around $z$ in the sense that the relative measure of $C$ converges to 1 as one “zooms in” on $z$. The theorem involves a null set of exceptions $z$. If $C$ is algorithmic in some sense, then the resulting null set is also algorithmic in the sense of the following definition.

\begin{definition}[\cite{miyabe_nies_zhang_2016}]\label{density}
(i) The lower Lebesgue density of a Lebesgue measurable set $P\subseteq \mathbb{R}$ at a point $z$ is 
$$\underline{\rho}(P|z)=\liminf_{\gamma,\delta \to 0^+}\frac{m([z-\gamma,z+\delta]\cap P)}{\gamma+\delta},$$
where $m$ is the Lebesgue measure on $\mathbb{R}$.

(ii) A real $z$ is a density-one point if $\underline{\rho}(P|z)=1$ for every effectively closed set $P$ containing $z$.

(iii) We say that $z$ is density random if $z$ is a density-one point and 1-random.
\end{definition}

Every weakly-2-random real is density random \cite{miyabe_nies_zhang_2016}. As mentioned in the introduction section, there are 1-random reals which are not density random.
\begin{definition}[Bienvenu et al. \cite{MR3474456}]\label{ice}
    A nondecreasing function $g:[0,1]\rightarrow\mathbb{R}$ is interval-c.e.\ if $g(0)=0$ and $g(x)-g(y)$ is a left-c.e.\ real uniformly in all rational pairs $x<y$.
\end{definition}

\begin{theorem}[Miyabe, Nies and Jing Zhang \cite{miyabe_nies_zhang_2016}]\label{25}
	$x\in [0,1]$ is density random if and only if $g'(x)$ exists for every interval-c.e. function $g$.
\end{theorem}

In particular, our function $f$ can be viewed as an interval-c.e.\ function and we shall use the theorem above in Section 3.

We also consider density in the setting of Cantor space.
\begin{definition}[\cite{miyabe_nies_zhang_2016}]
(i) Given a measurable set $C \subseteq 2^\omega$ and $x\in 2^\omega$. The lower density of $x$ in $C$ is defined to be 
$$\underline{\rho}_2(C|x)=\liminf_{n\to\infty}2^n \mu(C \cap [x\upharpoonright n]).$$

(ii) A real $x\in 2^\omega$ is a density-one point in Cantor space if $\underline{\rho}_2(P|x)=1$ for every $\Pi_1^0$ class $P$ containing $x$.

(iii) A real $z\in [0,1]$ is a dyadic density-one point if its dyadic expansion is a density one point in Cantor space.
    
\end{definition}

\begin{theorem}[Khan and Miller \cite{MR3471130}]\label{Khan}
    Let $z\in[0,1]$ be a 1-random dyadic density-one point. Then $z$ is a density-one point.
\end{theorem}

\subsection{Hausdorff dimension}
Our major reference on Hausdorff dimension is Mattila \cite{geometry}.
Given a set $A \subseteq \mathbb{R}$, let $\dim_H(A)$ be the Hausdorff dimension of $A$. The generalized Cantor set $C^\gamma$, where the scale $\gamma > 1$, is constructed by iteratively removing the open middle $\frac{1}{\gamma}$ of each remaining segment. Formally: $C_0^\gamma := [0,1]$,
    $$C_n^\gamma := \frac{(\gamma-1)}{2\gamma}C_{n-1}^\gamma \cup \left( \frac{\gamma+1}{2\gamma} + \frac{(\gamma-1)}{2\gamma}C_{n-1}^\gamma \right),$$
and $C^\gamma:= \bigcap_n C_n^\gamma$.

\begin{fact}[see \cite{geometry} 4.10] \label{gamma}
    $\dim_H(C^\gamma)=-\frac{\log2}{\log\frac{\gamma-1}{2\gamma}}$.
\end{fact}

\begin{fact}\label{Lip} Let $A \subseteq \mathbb{R}$ and $f \colon A \to \mathbb{R}$ satisfying the Lipschitz condition. Then $\dim_H(A) \ge \dim_H(f(A))$.
\end{fact}

\section{Analytic properties of the function and its range}
Recall our function \[f(x)=\sum_{\sigma\le_{L}x} 2^{-K(\sigma)}.\] In this section we discuss analysis properties of $f$. Let $$\hat{\Omega}(\sigma)=\sum\limits_{\sigma\preceq\tau}2^{-K(\tau)}.$$

\begin{lemma}[H\"olzl et al. \cite{hölzl_merkle_miller_stephan_yu_2020}]\label{22}
	If $x$ is 1-random, then
    $\varliminf\limits_{n\to\infty}2^n\hat{\Omega}(x\upharpoonright n)=0$.
\end{lemma}

It is clear that $f$ is continuous and $x<_Ly$ implies $f(x)<f(y)$. Now we give a characterization of the ``differentiable'' points of $f$. 

Since we define $f$ as a function from $2^\omega$ to $[0,1]$, the differentiability of such a function is not well-defined. We need to translate $f$ to a real function $F:[0,1]\to [0,1]$ as follows: let $b:[0,1]\to 2^\omega$ be the function mapping a real $\alpha \in [0,1]$ to its binary expansion in $2^\omega$. For the dyadic rationals less than 1, we choose its binary expansion which has infinitely many 0's. Then we define
$$F=f \circ b.$$
The function $F$ is strictly increasing, right continuous, and the discontinuous points of $F$ are exactly the dyadic rationals. Moreover, $F$ is interval-c.e.\ in the sense of Definition \ref{ice}.

Define an auxiliary function
    \begin{equation*}
    \begin{split}
                \hat{f}(\sigma)&=2^{|\sigma|}(f(\sigma1^\infty)-f(\sigma0^\infty))\\ 
                &=2^{|\sigma|}(\hat{\Omega}(\sigma)-\sum\limits_{m\ge0}2^{-K(\sigma0^m)})=2^{|\sigma|}\sum\limits_{\sigma0^\infty<_L\tau\le_L\sigma1^\infty}2^{-K(\tau)}.
    \end{split}
    \end{equation*}
    Then
	$$2\hat{f}(\sigma)-\hat{f}(\sigma0)-\hat{f}(\sigma1)
			=2^{|\sigma|+1}\sum\limits_{m\ge0}2^{-K(\sigma10^m)}\ge0,$$
	i.e. $\hat{f}$ is a c.e.\ supermartingale (see Section 2 for the definition). 

    \begin{lemma} \label{1111}
        For any $x\in 2^\omega$, $x$ is not 1-random if and only if $\varlimsup\limits_{n\to\infty}\hat{f}(x\upharpoonright n)=\infty$.
    \end{lemma}
    
    \begin{proof}
        Since $\hat{f}$ is a c.e.\ supermartingale, if $\varlimsup\limits_{n\to\infty}\hat{f}(x\upharpoonright n)=\infty$, then $x$ is not 1-random.
        
        On the other hand, suppose $x$ is not 1-random. Then
	\begin{equation}
		\hat{f}(x\upharpoonright n)=2^n\sum\limits_{(x\upharpoonright n)0^\infty<_L\tau\le_L(x\upharpoonright n)1^\infty}2^{-K(\tau)}\ge2^n\cdot 2^{-K((x\upharpoonright n)1)}\ge2^{n-K(x\upharpoonright n)-c}.
		\nonumber
	\end{equation}
Since $x$ is not 1-random, for each $d$ there is $n$ such that $n-K(x\upharpoonright n)>d$, hence $\varlimsup\limits_{n\to\infty}\hat{f}(x\upharpoonright n)=\infty$.
    \end{proof}

\begin{lemma}\label{32}
If $x$ is not 1-random, then $F$ is not differentiable at $x$.
\end{lemma}

\begin{proof}
Given $x$ which is not 1-random and suppose $F$ is differentiable at $x$. We may assume that $x$ is not a dyadic rational since $F$ is not continuous at dyadic rational points. Suppose $x=0.y$ for $y \in 2^\omega$. Define $y^n_i=(y\upharpoonright n)i^\infty$ for $i=0,1$. Let $x_i^n=0.y_i^n$, so $x_0^n$ and $x_1^n$ are the two $n$-bit dyadic rationals closest to $x$. Then 
	\begin{equation}
		\begin{split}
			\frac{F(x^n _1)-F(x)}{x^n _1-x}+\frac{F(x)-F(x^n _0)}{x-x^n _0}
            & \geq\frac{F(x^n _1)-F(x)}{x^n _1-x^n _0}+\frac{F(x)-F(x^n _0)}{x^n _1-x^n _0}\\
			&=\frac{F(x^n _1)-F(x^n_0)}{x^n _1-x^n _0}\\
            &\ge 2^{n}(f(y^n_1)-f(y^n_0))=\hat{f}(y\upharpoonright n).
			\nonumber
		\end{split}
	\end{equation}
    Then $\varlimsup\limits_{n\to\infty}\hat{f}(y\upharpoonright n) \leq 2F'(x) < \infty,$
              which is impossible by Lemma \ref{1111}.
\end{proof}

Given a string $\sigma\in 2^{<\omega}$. Suppose $m$ is the largest number less than $|\sigma|$ such that $\sigma(m)=0$ and $\sigma(j)=1$ for all $j>m$. Define $\hat{\sigma}=\sigma \upharpoonright m$ if $m$ exists.
\begin{lemma}\label{H} Suppose $x\in 2^\omega$ is 1-random. Then

	(i) $\lim\limits_{n\to\infty}2^n\sum\limits_{m\ge0}2^{-K((x\upharpoonright n)0^m)}=0$.

    (ii) $\lim\limits_{n\to\infty}2^n\sum\limits_{m\ge0}2^{-K(\widehat{(x\upharpoonright n)}10^m)}=0$.
\end{lemma}

\begin{proof}
	(i). Define $G(\sigma)=2^{|\sigma|}\sum\limits_{m\ge0}2^{-K(\sigma0^m)}$. Suppose (i) does not hold, there is a 1-random $x$ and $\epsilon>0$, such that there are infinitely many $n$ satisfying $G(x\upharpoonright n)>\epsilon$. Define $$E_n=\{\sigma:|\sigma|<n+1\land G(\sigma)>\epsilon\},$$ and let $E=\bigcup\limits_{n\in\omega}E_n$. Divide $E_n$ into at most $2^{n+1}$ many disjoint subsets as follows:\par
	Let $E_{n,0}=\emptyset$, if $E_{n,i}$ are defined for all $i<j$, let $\sigma_j$ be the lexicographically least in $E_n\big\backslash\bigcup\limits_{i<j}E_{n,i}$, and define $$E_{n,j}=\{\tau:\tau \in E_n\big\backslash\bigcup\limits_{i<j}E_{n,i} \land \sigma_j\preceq\tau\prec\sigma_j 0^\infty\}.$$
    For any distinct $\tau_1$ and $\tau_2$ in $E_{n,i}$, they are of different lengths. And if $|\tau_1|<|\tau_2|$, we have $\tau_1\prec\tau_2$. Then
	\begin{equation}
		\begin{split}
			S_n:=\sum\limits_{\sigma\in E_n}2^{-|\sigma|}\epsilon
			=\sum\limits_{i<2^{n+1}}\sum\limits_{\sigma\in E_{n,i}}2^{-|\sigma|}\epsilon
			&\le\sum\limits_{i<2^{n+1}}\sum\limits_{\sigma\in E_{n,i}}2^{-|\sigma|}G(\sigma_i)\\
			&=\sum\limits_{i<2^{n+1}}\sum\limits_{\sigma\in E_{n,i}}\bigg(2^{-|\sigma|+|\sigma_i|}\sum\limits_{m\ge0}2^{-K(\sigma_i 0^m)}\bigg)\\
			&\le\sum\limits_{i<2^{n+1}}2\cdot\sum\limits_{m\ge0}2^{-K(\sigma_i 0^m)}\le2.
			\nonumber
		\end{split}
	\end{equation}
	The last inequality holds because by the choice of $\sigma_i$, if $i \ne j$, then for all $m_0,m_1$, $\sigma_i0^{m_0} \ne \sigma_j0^{m_1}$.
    Then we have $S=\sum\limits_{\sigma\in E}2^{-|\sigma|}\cdot\epsilon=\lim\limits_{n\to\infty}S_n\le2$ and $\{(|\sigma|+1-\log \epsilon, \sigma)\}_{\sigma\in E}$ forms a KC set. Then $K(x\upharpoonright n)\le n+1-\log \epsilon+c$ for infinitely many $n$, a contradiction.

    (ii). Similar to (i). Define $H(\sigma)=2^{|\sigma|}\sum\limits_{m\ge0}2^{-K(\hat{\sigma}10^m)}$ if $\hat{\sigma}$ exists, otherwise define $H(\sigma)=0$. Suppose (ii) does not hold, then there is a 1-random $x$ and $\epsilon>0$, such that there are infinitely many $n$ satisfying $H(x\upharpoonright n)>\epsilon$. Define $$F_n=\{\sigma:|\sigma|<n+1\land H(\sigma)>\epsilon\},$$ and let $F=\bigcup\limits_{n\in\omega}F_n$. Divide $F_n$ into at most $2^{n+1}$ many disjoint subsets as follows:\par
	Let $F_{n,0}=\emptyset$, if $F_{n,i}$ are defined for all $i<j$, let $\sigma_j$ be the lexicographically least in $F_n\big\backslash\bigcup\limits_{i<j}F_{n,i}$, and define $$F_{n,j}=\{\tau:\tau \in F_n\big\backslash\bigcup\limits_{i<j}F_{n,i} \land \sigma_j\preceq\tau\prec\sigma_j 1^\infty\}.$$
    For any distinct $\tau_1$ and $\tau_2$ in $F_{n,i}$, they are of different lengths. And if $|\tau_1|<|\tau_2|$, we have $\tau_1\prec\tau_2$. And if $i \ne j$, then $\widehat{\sigma_i}10^{m_0} \ne \widehat{\sigma_j}10^{m_1}$ for all $m_0, m_1 \in \omega$. Then 
    $T_n:=\sum_{\sigma \in F_n}2^{-|\sigma|}\epsilon \le 2$ by the same proof as in (i). Then $\{(|\sigma|+1-\log \epsilon, \sigma)\}_{\sigma\in F}$ forms a KC set. Then $K(x\upharpoonright n)\le n+1-\log \epsilon+c$ for infinitely many $n$, a contradiction.
\end{proof}

\begin{lemma}\label{0}
    Suppose that $F$ is differentiable at $x$, then $F'(x)=0$.
\end{lemma}

\begin{proof}
    Suppose not, let $F'(x)=A>0$. Suppose $x=0.y$ for $y \in 2^\omega$. Define $y^n_i=(y\upharpoonright n)i^\infty$ for $i=0,1$, and let $x_i^n=0.y_i^n$. Then for any $\epsilon$ with $0<\epsilon<A$, there is $N$ such that for all $n>N$,
    $$\frac{F(x^n _1)-F(x)}{x^n _1-x},\frac{F(x)-F(x^n _0)}{x-x^n _0}>A-\epsilon>0.$$

    We may assume that $y \upharpoonright N \ne 1^N$, since $y$ is 1-random by Lemma \ref{32}. Then
\begin{equation}
	\begin{split}
            \hat{f}(y\upharpoonright n)&=2^{n}(f((y\upharpoonright n)1^\infty)-f((y\upharpoonright n)0^\infty))\\
            &=2^{n}(F(x_1^n)-F(x_0^n)-\sum_{j \ge 0}2^{-K(\widehat{(y\upharpoonright n)}10^j)})\\
		&=\frac{x_1^n-x}{x_1^n-x_0^n}\frac{F(x_1^n)-F(x)}{x_1^n-x}+\frac{x-x_0^n}{x_1^n-x_0^n}\frac{F(x)-F(x_0^n)}{x-x_0^n}-2^{n}\sum_{j \ge 0}2^{-K(\widehat{(y\upharpoonright n)}10^j)}\\
		&>A-\epsilon-2^{n}\sum_{j \ge 0}2^{-K(\widehat{(y\upharpoonright n)}10^j)}.
		\nonumber
	\end{split}
\end{equation}
Then $\varliminf\limits_{n\to\infty}\hat{f}(y\upharpoonright n) \ge A-\epsilon$ by Lemma \ref{H} (ii). However, since $y$ is 1-random,  $\varliminf\limits_{n\to\infty}\hat{f}(y\upharpoonright n)\le\varliminf\limits_{n\to\infty}2^n\hat{\Omega}(y\upharpoonright n)=0$ by Lemma \ref{22}, a contradiction.
\end{proof}

\begin{theorem}
A real $x\in [0,1]$ is density random if and only if $F$ is differentiable at $x$. In this case $F'(x)=0$.
\end{theorem}

\begin{proof}
Since $F$ is interval-c.e., it is differentiable at density random reals by Theorem \ref{25}. On the other hand, suppose that $x$ is not density random but $F$ is differentiable at $x$. Then $x$ is 1-random and $F'(x)=0$ by Lemma \ref{32} and \ref{0}.
By Theorem \ref{Khan}, $x$ is not a dyadic density-one point. Suppose $x=0.y$ where $y \in 2^\omega$. Then $y$ is not a density-one point in Cantor space. Then there is a $\Pi^0_1$ class $P$ and some $\epsilon>0$ such that $y\in P$ and $\varliminf\limits_{n\to\infty}2^n\mu([y\upharpoonright n]\cap P)<1-\epsilon$. Since $2^\omega\backslash P$ is $\Sigma^0_1$, take $W\subseteq 2^{<\omega}$ be c.e.\ prefix-free such that $[W]=2^\omega\backslash P$, where $[W]=\bigcup_{\sigma\in W}[\sigma]$. By the KC theorem for $\{(|\sigma|,\sigma)\}_{\sigma\in W}$, there is a constant $c$ such that $K(\sigma)\le|\sigma|+c$ for all $\sigma\in W$. Define $y^n_i=(y\upharpoonright n)i^\infty$ for $i=0,1$ and let $x^n_i=0.y^n_i$. By Lemma \ref{H}, choose $N$ such that for any $n>N$,
$$\frac{F(x^n_1)-F(x)}{x^n_1-x}+\frac{F(x)-F(x^n_0)}{x-x^n_0}+2^n\sum\limits_{m\ge0}2^{-K((y\upharpoonright n)0^m)}< 2^{-c-1}\epsilon.$$
Fix an $n>N$ such that $\mu([y\upharpoonright n]\cap[W])\ge2^{-n}\epsilon$. Then
	\begin{equation}
		\begin{split}
			2^{-c-1}\epsilon&>\frac{F(x^n_1)-F(x)}{x^n_1-x}+\frac{F(x)-F(x^n_0)}{x-x^n_0}+2^n\sum\limits_{m\ge0}2^{-K((y\upharpoonright n)0^m)}\\
            & \ge 2^n(F(x^n_1)-F(x)+F(x)-F(x^n_0))+2^n\sum\limits_{m\ge0}2^{-K((y\upharpoonright n)0^m)}\\
			&\ge2^n\sum\limits_{\tau\succeq y\upharpoonright n}2^{-K(\tau)}\\
			&\ge2^n\sum\limits_{\tau\succeq y\upharpoonright n\land\tau\in W}2^{-K(\tau)}\ge2^n\sum\limits_{\tau\succeq y\upharpoonright n\land\tau\in W}2^{-|\tau|-c}\ge2^n\ 2^{-n-c}\epsilon=2^{-c}\epsilon,
			\nonumber
		\end{split}
	\end{equation}
a contradiction.
\end{proof}

Next we study $f(2^\omega)$ from the perspective of mathematical analysis. A class of reals $P$ is \textit{nowhere dense} if its closure has empty interior. A class of reals $P$ is \textit{perfect} if it is closed and does not contain any isolated point, i.e.\ every real in $P$ is a limit of distinct reals in $P$. 
\begin{proposition}\label{null}
	$f(2^\omega)$ is null, nowhere dense and perfect.
\end{proposition}

\begin{proof}
	For all $\sigma \in 2^{<\omega}$, define $I_\sigma=(f(\sigma01^\infty),f(\sigma10^\infty))$. Then $I_\sigma \cap f(2^\omega)= \emptyset$. Let $\Omega=\Omega_U=\sum_\sigma 2^{-K_U(\sigma)}$ where $U$ is the optimal machine used to define $f$. Since $f(2^\omega)\subseteq[\sum\limits_{m\ge0}2^{-K(0^m)},\Omega]$ and $I_\sigma$ are pairwise disjoint,
	\begin{equation}
		\begin{split}
			\mu(f(2^\omega))&\leq \Omega-\sum\limits_{m\ge0}2^{-K(0^m)}-\sum\limits_{\sigma\in2^{<\omega}}\mu(I_\sigma)\\
			&=\Omega-\sum\limits_{m\ge0}2^{-K(0^m)}-\sum\limits_{\sigma\in2^{<\omega}}(f(\sigma10^\infty)-f(\sigma01^\infty))\\
			&=\Omega-\sum\limits_{m\ge0}2^{-K(0^m)}-\sum\limits_{\sigma\in2^{<\omega}}\sum\limits_{m\ge0}2^{-K(\sigma10^m)}\\
			&\leq\Omega-\Omega=0.
			\nonumber
		\end{split}
	\end{equation}
	Then $f(2^\omega)$ is null.\par
    Since $f$ is continuous, $f(2^\omega)$ is closed. Since $f(2^\omega)$ is null and any non-empty open set is not null, $f(2^\omega)$ has empty interior, therefore nowhere dense. \par
	Since injective continuous functions map perfect sets to perfect sets and $f$ is such a function, $f(2^\omega)$ is perfect. 
\end{proof}

Given any prefix-free machine (not necessarily optimal) $M$, we define 
$f_M:x\mapsto \sum_{\sigma\le_L x}2^{-K_M(\sigma)}$ as before. The same proof as above shows that $f_M(2^\omega)$ is null.

In the rest of this section we show $\dim_H(f(2^\omega))=1$ by constructing a sequence of functions from $f(2^\omega)$ to a generalized Cantor set $C_\gamma$ (see subsection 2.3 for the definition). Each function satisfies the Lipschitz condition and hence preserves the Hausdorff dimension.

\begin{theorem}\label{hausdorff}
$\dim_H(f(2^\omega))=1$.
\end{theorem}

We shall use the following fact in the proof of the theorem.
\begin{fact}\label{0110}
    If $f(0^\infty)<x<f(1^\infty)$ and $x \notin f(2^\omega)$, then there is $\sigma \in 2^{<\omega}$ such that 
    $$x \in (f(\sigma01^\infty),f(\sigma10^\infty)).$$
\end{fact}

\begin{proof}
    Suppose not. For any $\sigma$, define $$I^0_\sigma=[f(\sigma0^\infty),f(\sigma01^\infty)] \text{ and }I^1_\sigma=[f(\sigma10^\infty),f(\sigma1^\infty)].$$
    Then there is a sequence of strings $\{\sigma_n\}_{n \in \omega}$ such that $|\sigma_n|=n$ and $\sigma_n \prec \sigma_{n+1}$ for all $n$ and a 0-1 sequence $\{a_n\}_{n \in \omega}$ such that $x \in I^{a_n}_{\sigma_n}$. Then define $y=\cup_n \sigma_n$. Since $f$ is continuous, and $|I^{a_n}_{\sigma_n}| \to 0$ when $n \to \infty$, we have $f(y)=x$, a contradiction.
\end{proof}

\begin{remark}\label{0311}
From Fact \ref{0110}, if we let $A_\sigma=(f(\sigma01^\infty),f(\sigma10^\infty))$, then we have 
$$[0,1]-f(2^\omega)=[0,f(0^\infty))\cup\bigcup_{\sigma \in 2^{<\omega}} A_\sigma\cup(f(1^\infty),1].$$
Intuitively, this means $f(2^\omega)$ can be constructed in a similar way as the Cantor ternary set, which is the basic idea of the proof below.
\end{remark}

\begin{proof} (of Theorem \ref{hausdorff}).
 Fix a scale $\gamma>3$. Recall that $C^\gamma= \bigcap_n C_n^\gamma$ as in subsection 2.2. We view $C_n^\gamma$ as the $n$-th level in the construction of $C^\gamma$ and suppose $C_n^\gamma=\bigcup_{i \le 2^{n}}[a_{n,i},b_{n,i}]$, where $$a_{n,0}<b_{n,0}<a_{n,1}<\dots<a_{n,2^{n}-1}<b_{n,2^{n}-1}.$$
Similarly let $f_n=\bigcup_{\sigma \in 2^{n}}[f(\sigma0^\infty),f(\sigma1^\infty)]$, so $f(2^\omega)=\bigcap_n f_n$ using Remark \ref{0311}, and we regard $f_n$ as the $n$-th level in the construction of $f(2^\omega)$.

Then for any $n$, we construct a monotone increasing piecewise linear function $g_n:[0,1]\to [0,1]$ such that each interval $[f(\sigma0^\infty),f(\sigma1^\infty)]$ is linearly mapped to $[a_{n,i},b_{n,i}]$ where $\sigma \in 2^n$ is the $i$-th interval counting from the left. And the gap between two intervals in $f_n$ is mapped to the gap between the corresponding intervals in $C_n^\gamma$ in a natural way.  

By Fact \ref{0110}, $g_n(x)$ are defined for all $n$ if and only if $x \in f(2^\omega)$. We claim that there is a constant $c$ such that for each $n\in \mathbb{N}$ and each $x,y\in f(2^\omega)$, $$|g_n(x)-g_n(y)|\le c|x-y|.$$ 

 Suppose the claim holds. Let $h_n=(\frac{\gamma-1}{2\gamma})^n$ be the length of every interval in $C_n^\gamma$. Then $h_n \to 0$ as $n \to \infty$. Hence if $x \in f(2^\omega)$, then $\lim_n g_n(x)$ exists.  Define a function $g: f(2^\omega) \to [0,1]$ by $g(x)=\lim_n g_n(x)$. Then $g$ is a monotone increasing bijection from $f(2^\omega)$ to $C^\gamma$. The function $g$ satisfies the Lipschitz condition since given $x,y \in f(2^\omega)$ and for any $\epsilon>0$ there is some $n$ such that
    $$ |g(x)- g(y)| \le |g_n(x)-g_n(y)|+\epsilon \le c|x-y| + \epsilon, $$ where $c$ is the constant in the claim above and hence $g$ is also Lipschitz. By Fact \ref{gamma} and \ref{Lip}, $\dim_H(f(2^\omega))\ge -\frac{\log2}{\log\frac{\gamma-1}{2\gamma}}$. Since $\gamma >3$ is arbitrary, let $\gamma \to \infty$, then $\dim_H(f(2^\omega))=1$.

 It remains to prove the claim. Since each $g_n$ is piecewise linear. The Lipschitz constant is only determined by the gradient of each piece of the maps. 
 
 The linear map from $[f(\sigma0^\infty),f(\sigma 1^\infty)]$ to $[a_{n,i},b_{n,i}]$ where $\sigma \in 2^n$ is the $i$-th interval counting from the left has gradient 
   \begin{equation}
        \begin{split}
            \frac{h_n}{f(\sigma 1^\infty)-f(\sigma0^\infty)}
            \le \frac{(\frac{\gamma-1}{2\gamma})^n}{2^{-K(\sigma1)}}&\le c_1(\frac{\gamma-1}{2\gamma})^n\cdot 2^{|\sigma|+2\log|\sigma|}\\
    &=c_1(\frac{\gamma-1}{2\gamma})^n\cdot 2^{n+2\log n}\le c_1 (\frac{\gamma-1}{\gamma})^n\cdot n^2 \le c_2
        \end{split}
        \nonumber
    \end{equation}
when $n$ is sufficiently large.

Similarly, the linear map from the gap between two intervals firstly appeared in the $n$-th level has gradient 
$$\frac{h_{n-1}-2h_n}{f(\sigma 10^\infty)-f(\sigma01^\infty)}
            \le \frac{(\frac{\gamma-1}{2\gamma})^{n-1}-2(\frac{\gamma-1}{2\gamma})^{n}}{2^{-K(\sigma1)}}
    \le c_3\frac{1}\gamma{}(\frac{\gamma-1}{2\gamma})^{n-1}\cdot 2^{|\sigma|+2\log|\sigma|}\le c_4$$
for some $\sigma \in 2^{n-1}$ and sufficiently large $n$.
\end{proof}

\section{The complexity of the function values}

Recall the Omega operator $x\mapsto\Omega^x=\sum_{V^x(\sigma)\downarrow}2^{-|\sigma|}$ defined in \cite{downey2005relativizing}, where $V$ is a fixed universal prefix-free oracle machine. It is widely known that the following properties holds for $\Omega$: \begin{enumerate}[(i)]
    \item $\Omega^x$ is left-$x$-c.e.; 
    \item $\Omega^x$ is $x$-random; 
    \item $x\oplus\Omega^x\equiv_T x'$. 
\end{enumerate}

In this section we investigate if the corresponding properties holds for $f$. 

Property (i) also holds for $f$: 

\begin{fact}\label{41}
    Given a real $x$, then $f(x)$ is left-$x$-c.e.\ and $x$ is right-$f(x)$-c.e.. Moreover, if $x$ is a left-c.e.\ real, then $f(x)$ is a left-c.e.\ real.
    \end{fact}

\begin{proof}
It is clear that $f(x)$ is left-$x$-c.e.\ and hence $x$ is right-$f(x)$-c.e.\ since $f$ is monotone increasing.
\end{proof}

Property (ii) does not generally hold for $f$, as the following simple observation shows that $f(x)$ cannot be arbitrarily random. In Section 5 we will further investigate the randomness properties of $f(x)$, see Theorem \ref{lizhang} and Theorem \ref{0418}. 

\begin{fact}\label{02}
        $f(2^\omega)$ is a $\Pi^{0}_1(\emptyset')$ class. Hence every real in $f(2^\omega)$ is not $\emptyset'$-Kurtz random.
\end{fact}

\begin{proof}
In Remark \ref{0311}, each interval $A_\sigma$ is uniformly $\Sigma^{0,\emptyset'}_1$ since $f(\sigma i^\infty)$ are  uniformly left-c.e.\ where $\sigma \in 2^{<\omega}$ and $i=0,1$.
\end{proof}

Property (iii) for $f$ also generally fails, since by (i) in the following Theorem, if $x\geq_T\emptyset'$ then $f(x)$ is computable in $x$. Yet we are able to obtain other Turing degree estimations regarding the relation of $x$ and $f(x)$. 

\begin{theorem}\label{402}
Given a real $x$.

(i) $x'\ge_T\emptyset'\oplus x\ge_T f(x)$;

(ii) $f(x)'\ge_T\emptyset'\oplus f(x)\ge_T x$;

(iii) $f(x)\oplus x\ge_T \emptyset'$.
\end{theorem}

\begin{proof}
(i) and (ii).  $\{{f(\sigma i^\infty)}\}_{\sigma \in 2^{<\omega},i=0,1}$ are uniformly $\emptyset'$-computable. 

	(iii). For every $n$, if $n\in\emptyset'_{s_n+1}-\emptyset'_{s_n}$ for some $s_n$, pick the least $m_n<4^n$ such that $\sum_{|\sigma|=m_n}2^{-K_{s_n+1}(\sigma)}\le4^{-n}$. This $m_n$ exists since $\sum_{\sigma}2^{-K_{s_n+1}(\sigma)}\le1$. Enumerate $(K_{s_n+1}(\sigma)-n+1,\sigma)$ for all $|\sigma|=m_n$ in a set $V$. Then $V$ is a KC set since 
    $$\sum\limits_{n\ge0}\sum\limits_{|\sigma|=m_n}2^{-K_{s_n+1}(\sigma)+n-1}=\sum\limits_{n\ge0}2^{n-1}\sum\limits_{|\sigma|=m_n}2^{-K_{s_n+1}(\sigma)}\le2^{-1}\sum\limits_{n\ge0}2^n\cdot4^{-n}=1.$$
	
	Define $f_s(x)=\sum_{\sigma\le_L x\land|\sigma|\leq s}2^{-K_s(\sigma)}$. We claim that for almost all $n$, if $f(x)-f_s(x)<2^{-8^n}$ at stage $s>4^n$, then $n\in\emptyset'$ if and only if $n\in\emptyset'_{s+1}$.\par
	Suppose not, there are infinitely many $n$ and $s>4^n$ such that $f(x)-f_s(x)<2^{-8^n}$ but $n\in\emptyset'_{s_n+1}-\emptyset'_{s_n}$ for some $s_n\ge s+1$. Then for any such $n$, for every $y\le_Lx$ we have $K_V(y\upharpoonright m_n)\le K_{s_n+1}(y\upharpoonright m_n)-n+1$, and 
	$K(y\upharpoonright m_n)\le 2m_n< 2\cdot4^n <8^n$.
	 Since $f(x)-f_s(x)<2^{-8^n}$, $K_{s_n+1}(y\upharpoonright m_n)=K(y\upharpoonright m_n)$ since $s>4^n>m_n$. Then we have $\varlimsup\limits_{n\rightarrow\infty}K(y\upharpoonright m_n)-K_V(y\upharpoonright m_n)=\infty$, which is impossible. Hence $f(x)\oplus x\ge_T \emptyset'$.
\end{proof}

\begin{corollary}\label{comp}
     If $f(x)$ is computable, then $x$ is right-c.e.\ and Turing complete.
\end{corollary}

It is not known whether there is $x$ such that $f(x)$ is computable (or just right-c.e.), see the discussion in Subsection 6.2. However, by the following proposition, $f(x)$ is not right-c.e.\ for almost all $x$.
\begin{proposition}[with Slaman]
    If $f(x)$ is right-c.e., then $x$ is not 1-random.
\end{proposition}
\begin{proof}
     Given a right-c.e. real $q$ with approximation $(q_s)_{s\in \omega}$, we construct a prefix-free machine $M$ with coding constant $c_M$ ($(\forall\sigma) [K(\sigma)\le K_M(\sigma)+c_M$]) by the Machine Existence Theorem. By the Recursion Theorem, $c_M$ can be used in the construction of $M$. 
     
     Let $J_\sigma = (f(\sigma0^\infty),f(\sigma1^\infty))$ and $J_{\sigma,s} = (f_s(\sigma0^\infty),f_s(\sigma1^\infty))$. Whenever $q_s\in J_{\sigma,s}$ and $J_{\sigma,s}$ is sufficiently small, $M$ gives a short description of some string $\tau<_L\sigma$ to force $U$ give a short description of $\tau$ later, which implies $f(\sigma0^\infty)-f_s(\sigma0^\infty)$ is large enough to make $q\not\in J_\sigma$.

     \
    
     Construction of $M$ by a KC set $W$: 
     At stage $s+1$. If there is $\sigma$ such that $q_s\in J_{\sigma,s}$ and $\mu(J_{\sigma,s})\le 2^{-|\sigma|-c_M-1}$, choose the least $\sigma$ which has not been used and enumerate $(|\sigma|,\tau)$ into $W$, where $\tau$ is the shortest string such that $\tau\le_L \sigma$ and $K_s(\tau)\ge|\sigma|+c_M+1$. Declare that all the strings extending $\sigma$ have been \textit{used}. 

     \

     Since we only put $(|\sigma|, \tau)$ into $W$ if $\sigma$ has not been used, the weight of $W$ is no more than $1$. 
     
     Suppose there is some 1-random $x$ such that $f(x)=q$, then there is some $n$ such that $\mu (J_{x\upharpoonright n})< \hat\Omega(x\upharpoonright n)< 2^{-n-c_M-1}$ by Lemma \ref{22}. Then there is some stage $s$ such that $q\in J_{x\upharpoonright n,s}$ and $\mu(J_{x\upharpoonright n,s})\le 2^{-n-c_M-1}$. By the construction of $M$, $f(\sigma0^\infty)-f_s(\sigma0^\infty)\ge 2^{-n-c_M-1}$ and $q\not\in J_{x\upharpoonright n}$, a contradiction.
\end{proof}

\section{The randomness of the function values}

Recall that $\Omega^x$ is always $x$-random, while we have shown that $f(x)$ cannot be arbitrarily random. In this section, we investigate how random is the image $f(x)$. We first give positive results by characterizing all cases where $f(x)$ is $x$-random, and derive corollaries from this result. We then also give negative results showing that there are cases where $f(x)$ is not even 1-random. 

\subsection{Positive results}

The following is the main theorem for the randomness properties of $f$. It gives a characterization of the weakly low for $K$ reals. It can also be viewed as a generalization of the classical result that $\Omega$ is 1-random since for any computable (and hence weakly low for $K$) real $x$, $f(x)$ is $x$-random. Many corollaries follow from this theorem.
\begin{theorem} \label{lizhang}
            For any real $x$, $x$ is weakly low for $K$ if and only if $f(x)$ is $x$-random.
\end{theorem}

\begin{proof}
Let $e:\omega \rightarrow \{\sigma \in 2^{<\omega}:\sigma\le_L x \}$ be an $x$-computable bijection. Define a right-$x$-c.e.\ function $g$ by $g(n)=K(e(n))$. Then $f(x)=\sum_n2^{-g(n)}$.

$\Rightarrow:$
Suppose that $x$ is  weakly low for $K$. Since $e$ is $x$-computable, $K^x (n)\le^+ K(e(n))$. We claim that there are infinitely many $n$ such that $ g(n)\le^+ K^x(n)$. And hence $g$ is an $x$-Solovay function (see Section 2 for the definition) and $f(x)=\sum_n2^{-g(n)}$ is $x$-random by Theorem \ref{Solo}. 

In fact, there are infinitely many $s$ such that $K(0^s) \le^+ K^x(0^s)$ since $x$ is weakly low for $K$.
Hence there are infinitely many $n$ such that $e(n)$ is of the form $0^s$ and $$ g(n)=K(0^s)\le^+K^x(0^s)= K^x (e(n))\le^+ K^x(n).$$ 

$\Leftarrow:$
   Suppose $f(x)$ is $x$-random. By Theorem \ref{Solo}, $g$ is an $x$-Solovay function. Then there are infinitely many $n$ such that $$K(e(n))=g(n) \le^+ K^x(n)\le^+K^x (e(n)).$$
   So $x$ is weakly low for $K$. 
\end{proof}

\begin{corollary}
     Suppose $x<_Ly$. Then $x \oplus y$ is weakly low for $K$ if and only if $f(y)-f(x)$ is $x \oplus y$-random.
\end{corollary}

\begin{proof}
    The proof is the same as Theorem \ref{lizhang}.
    \end{proof}

\begin{corollary}\label{411}
    $\{ x: f(x) \text{ is not 1-random} \}$ is null.  
\end{corollary}

\begin{proof}
    Let $x$ be 2-random. Then by Theorem \ref{Miller}, $x$ is weakly low for $K$. By Theorem \ref{lizhang}, $f(x)$ is $x$-random and therefore 1-random. This shows that 
    $$\{ x: f(x) \text{ is not 1-random} \} \subseteq \{ x: x\text{ is not 2-random} \}. $$
    The latter one is null.
\end{proof}

\begin{corollary}\label{0409}
    Suppose that $x$ is 2-random, then $f(x)$ and $x$ are Turing incomparable.
\end{corollary}

\begin{proof}
    $x$ is weakly low for $K$ by Theorem \ref{Miller}. Hence $f(x)$ is $x$-random. Since $x$ is 1-random, by van Lambalgen's theorem (see \cite{MR2548883} Theorem 3.4.6), $x$ is $f(x)$-random. Hence $x$ and $f(x)$ are Turing incomparable.
\end{proof}

\begin{corollary}\label{Delta02WithRandomImageImpliesKTrivial}
Let $x \le_T \emptyset'$ and let $f(x)$ be $x$-random. Then $x$ is $K$-trivial.
\end{corollary}

\begin{proof}
$x$ is weakly low for $K$ by Theorem \ref{lizhang}. Then $x$ is low for $\Omega$ by Theorem \ref{Miller}. Since $x$ is both $\Delta_2^0$ and low for $\Omega$, then $x$ is $K$-trivial by \cite{MR2548883} exercise 5.1.27. 
\end{proof}

\begin{corollary}
If $x$ is weakly low for $K$ but not $K$-trivial (for example, if $x$ is 2-random). Then $f(x) \ngeq_T \emptyset'$.
\end{corollary}

\begin{proof}
    Suppose $f(x)\ge_T\emptyset'$, then $x\le_T\emptyset'$ by Theorem \ref{402}. Also $f(x)$ is $x$-random by Theorem \ref{lizhang}. By Corollary \ref{Delta02WithRandomImageImpliesKTrivial}, $x$ is $K$-trivial, a contradiction. 
\end{proof}

\begin{corollary}\label{407}
    For every $K$-trivial $x$, $f(x)$ is a left-c.e.\ real.
\end{corollary}

\begin{proof}
  By Fact \ref{41}, $f(x)$ is left-$x$-c.e.\ and therefore difference left-$x$-c.e.. By Theorem \ref{dce}, $x$ is low for difference left-c.e.\ reals, so $f(x)$ is difference left-c.e.. Also, since $x$ is $K$-trivial, $x$ is weakly low for $K$, and by Theorem \ref{lizhang} $f(x)$ is $x$-random. In particular $f(x)$ is 1-random. Now by Theorem \ref{cecoce}, $f(x)$ is either left-c.e.\ or right-c.e.. If $f(x)$ is right-c.e., then since it is also left-$x$-c.e., it is computable in $x$, which is impossible. Hence $f(x)$ is a left-c.e.\ real.
\end{proof}

\begin{corollary}\label{KTrivialImpliesOmegaImage}
    If $x$ is $K$-trivial, then $f(x)$ is an Omega number and therefore $f(x) \equiv_{wtt} \emptyset'$. In particular, if $x \equiv_T y$ are $K$-trivial, then $f(x) \equiv_{wtt} f(y)$.
\end{corollary}

\begin{proof}
    Given a $K$-trivial real $x$. By Corollary \ref{407}, $f(x)$ is a left-c.e.\ real. By Theorem \ref{lizhang}, $f(x)$ is $x$-random and hence 1-random. Therefore $x$ is an Omega number, and by Theorem \ref{KS}, $f(x)\equiv_{wtt} \emptyset'$.
\end{proof}
The $K$-trivial reals form an ideal in the $\Delta_2^0$ wtt-degrees (see \cite{MR2548883} Theorem 5.2.17). Hence $f$ is weak truth-table invariant in the ideal induced by the $K$-trivial reals. On the other hand, the following corollary shows that $f$ is not a Turing invariant function in general.

\begin{corollary}[with Slaman]
         If $x$ is 2-random then $f(x)\not\equiv_T f(\bar{x})$, where $\bar{x}$ is defined by $\bar{x}(i)=1-x(i)$ for all $i \in \omega$. In particular, there are $x \equiv_T y$ such that $f(x) \not\equiv_T f(y).$
    
\end{corollary}

    \begin{proof}
        Let $x$ be 2-random. If $f(x)\equiv_T f(\bar{x})$, then by Fact \ref{41}, $\bar{x}$ is right-$f(\bar{x})$-c.e.\ and therefore $x$ is left-$f(x)$-c.e.. Also by Fact \ref{41}, $x$ is right-$f(x)$-c.e., so $x\le_Tf(x)$. This contradicts Corollary \ref{0409} and that $x$ is 2-random. Therefore $f(x)\not\equiv_T f(\bar{x})$. 
    \end{proof}

    \begin{corollary}\label{0416}
        If $x$ is left-c.e., then $f(x)$ is an Omega number and therefore $f(x)\equiv_{wtt}\emptyset'$. In particular, if $x\equiv_T y$ are left-c.e., then $f(x)\equiv_Tf(y)$. 
    \end{corollary}
    
\begin{proof}
    By Corollary \ref{KTrivialImpliesOmegaImage}, $f(0^\omega)$ is an Omega number. Since $x$ is left-c.e., $f(x)-f(0^\omega)$ is also left-c.e.. Then $f(x)$ is the sum of an Omega number and a left-c.e.\ real, so by \cite{MR2548883} Proposition 3.2.27), $f(x)$ is also an Omega number. 
\end{proof}

\subsection{Negative results}

Recall that from Fact \ref{02}, every real in the range of $f$ is not $\emptyset'$-Kurtz random. Moreover, we can now construct a real $x$ such that $f(x)$ is not 1-random. On the other hand, the set of such reals is null (Corollary \ref{411}).

\begin{theorem}[with Yu]\label{0418}
  There is an $x$ such that $f (x)$ is not 1-random.
\end{theorem}

A real $x$ is \textit{normal} if for each $k$, all strings $\sigma$ of length $k$ appear in the array $\{x\upharpoonright[nk,(n+1)k)\}_{n\in\omega}$ with equal frequency. The proof uses an argument in the spirit of Becher et al. \cite{MR2275867} to make $f(x)$ not normal and hence not 1-random. We construct this real $x$ by $\omega$ many stages and at each stage we use the perturbation lemma below, which says that the size of perturbation can be controlled from below and above. Fix a constant $c>1$ such that for all $\sigma$ and $n$,
\[ K (\sigma 010^n), K (\sigma 101^n) \leq K (\sigma 10^n) + c - 2. \]

\begin{lemma}[Small Perturbation]
    \label{Lemma1}For all reals $x$ and $n\in\omega$, if there is $z$ such that $|f(x)-f(z)|>2^{-n}$, then there is $y$ such that $2^{-n-c}<|f(x)-f(y)|<2^{-n}$. Moreover, we can choose $y$ so that $f(y)$ is 1-random. 
\end{lemma}

\begin{proof}
    Consider \begin{align*}
        A^- &= \{y:f(y)\leq f(x)-2^{-n}\} \\
        B^- &= \{y:f(x)-2^{-n}<f(y)<f(x)-2^{-n-c}\} \\
        A &= \{y:f(x)-2^{-n-c}\leq f(y)\leq f(x)+2^{-n-c}\} \\
        B^+ &= \{y:f(x)+2^{-n-c}<f(y)<f(x)+2^{-n}\} \\
        A^+ &= \{y:f(x)+2^{-n}\leq f(y)\}. 
    \end{align*} Then $2^\omega$ is the disjoint union of these 5 sets, and elements in these sets are increasing, i.e.\ every $y$ in $A^-$ is strictly smaller than any $y$ in $B^-$, etc.. Since $f$ is continuous, $A^-$, $A$ and $A^+$ are closed and $B^-$ and $B^+$ are open. Any real $y$ in $B^-\cup B^+$ witnesses the Lemma, so it suffices to prove that $B^-\cup B^+$ is non-empty. Then since $B^-\cup B^+$ is open, we can take $y$ to be computable, so by Theorem \ref{lizhang} $f(y)$ is 1-random. 

    For a contradiction, suppose that $B^-$ and $B^+$ are both empty. Then $A^-$, $A$ and $A^+$ are clopen. Since $x\in A$, $A$ is non-empty, and since $z$ is in either $A^-$ or $A^+$, at least one of $A^-$ and $A^+$ is non-empty. We distinguish 3 cases. \begin{enumerate}[(1)]
        \item $A^-$ is empty, $A=[0^\infty,\sigma 01^\infty]$ and $A^+=[\sigma 10^\infty,1^\infty]$ for some $\sigma$. 
        \item $A^-=[0^\infty,\tau 01^\infty]$, $A=[\tau 10^\infty,1^\infty]$ for some $\tau$ and $A^+$ is empty. 
        \item $A^-=[0^\infty,\tau 01^\infty]$, $A=[\tau 10^\infty,\sigma 01^\infty]$ and $A^+=[\sigma 10^\infty,1^\infty]$ for some $\sigma$, $\tau$. We distinguish 3 subcases by (3a) $|\sigma|>|\tau|$, (3b) $|\sigma|=|\tau|$ or (3c) $|\sigma|<|\tau|$. 
    \end{enumerate}

    In case (1), (3a) or (3b), $\inf A<_L\sigma 010^m<_L\sup A$ for all $m$. Then \begin{align*}
        f(\sigma 10^\infty)-f(\sigma 01^\infty) &= \sum_{m\geq 0}2^{-K(\sigma 10^m)} \le 2^{c-2}\sum_{m\geq 0}2^{-K(\sigma 010^m)} \\
        &\le 2^{c-2}(f(\sup A)-f(\inf A))\le 2^{c-2}\cdot 2^{-n-c+1}=2^{-n-1}.
    \end{align*} But also $f(\sigma 10^\infty)-f(\sigma 01^\infty)\geq 2^{-n}-2^{-n-c}$, a contradiction. 

    In case (2) or (3c) (in fact, also (3b)), $\inf A<_L\tau 101^m<_L\sup A$ for all $m$. Then similarly \[f(\tau 10^\infty)-f(\tau 01^\infty)=\sum_{m\geq 0}2^{-K(\tau 10^m)} \le 2^{c-2}\sum_{m\geq 0}2^{-K(\tau 101^m)}\leq 2^{-n-1}\] and $f(\tau 10^\infty)-f(\tau 01^\infty)\geq 2^{-n}-2^{-n-c}$, a contradiction. 

    This shows that $B^-\cup B^+$ is non-empty, and the Lemma is proved. 
\end{proof}

\begin{proof}(of Theorem \ref{0418}).
    We prove that there is $x$ such that $f(x)$ is not normal. 

    We divide $I$ into disjoint union of intervals $I_s=[a_s,a_{s+1})$ where $a_s=s(c+1)$, so each interval is of length $c+1$. We fix a large $m$ such that $2^{-a_m}<f(1^\infty)-f(0^\infty)$. We define a sequence $(x_n)_{n\in\omega}$ and let $x=\lim_n x_n$. Let $x_0$ until $x_m$ be arbitrary. To define $x_{s+1}$ from $x_s$ for $s\geq m$, we distinguish two cases: 
    
    If $f(x_s)\upharpoonright I_s=10^c$, then we apply Lemma \ref{Lemma1} to $x_s$ and $n=a_s+1$ by verifying either $0^\infty$ or $1^\infty$ will qualify for the required $z$. We get $y$ such that $2^{-a_{s+1}}<|f(y)-f(x_s)|<2^{-a_s-1}$, therefore $f(y)\upharpoonright I_t=f(x_s)\upharpoonright I_t$ for all $t<s$, and $f(y)\upharpoonright I_s\neq f(x_s)\upharpoonright I_s=10^c$. Let $x_{s+1}=y$. 

    If $f(x_s)\upharpoonright I_s\neq 10^c$, then we do nothing and let $x_{s+1}=x_s$. 

    Given any $s$, $f(x_t)\upharpoonright I_s$ stays the same for all $t>\max\{s,m\}$. Therefore $f(x_s)$ converges, and thus $x_s$ converges to some $x$. Given any $s\geq m$, $f(x_t)\upharpoonright I_s\neq 10^c$ for all $t>s$, so $f(x)\upharpoonright I_s\neq 10^c$, and therefore $f(x)$ is not normal. 
\end{proof}

\begin{remark}\label{remark0418}
   The proof of Theorem \ref{0418} can be modified to be effective in $\emptyset'$. That is we start with a computable real $x_m$. To define $x_{s+1}$ from $x_s$ for $s\geq m$, $\emptyset'$ can decide the case distinction, and when applying Lemma \ref{Lemma1}, $\emptyset'$ can search for a $\sigma$ with $[\sigma]\subseteq B^-\cup B^+$, and we can take $x_{s+1}=\sigma 0^\infty$. Since we fixed the first $a_s$ bits of $f(x_t)$ for all $t\geq\max\{s,m\}$, $f(x)\leq_T\emptyset'$. Then by Theorem \ref{402} (ii), $x\leq_T\emptyset'$. Therefore we get a real $x \le_T \emptyset'$ such that $f(x)$ is not 1-random.
\end{remark}

The perturbation not only attacks the randomness. It can also be used to split the reals. Hence we can construct a perfect tree of reals in $f(2^\omega)$ such that none of the reals is 1-random.
\begin{corollary}
  There are $2^{\aleph_0}$ many reals $x$ such that $f (x)$ is not 1-random.
\end{corollary}

\begin{proof}
  Fix $x_\lambda$ ($\lambda$ is the empty string) such that $x_\lambda(0)=0$ and $f(x_\lambda)$ is 1-random. Choose $M$ as in the proof of Theorem \ref{0418}. Now we construct $2^{\aleph_0}$ many reals in $f(2^\omega)$:

  \

  At stage 1. Find $n>M$ and $y$ such that $f(x_\lambda)-f(y)>2^{-2(c+1)n}$ and $f(x_\lambda)\upharpoonright [n,n+2c+2)=10^c10^c$. As in the proof of Theorem \ref{0418}, find different $x_0,x_1\neq x_\lambda$ such that $f(x_0), f(x_1)$ is 1-random (see the moreover part of Lemma \ref{Lemma1}) and
  
  $$f(x_1)\upharpoonright [0,n)= f(x_0) \upharpoonright [0,n)=f(x_\lambda)\upharpoonright [0,n),$$
$$f(x_0) \upharpoonright [n,n+c+1)\neq f(x_\lambda)\upharpoonright [n,n+c+1),$$
$$f(x_1) \upharpoonright [n+c+1,n+2c+2)\neq f(x_\lambda)\upharpoonright [n+c+1,n+2c+2),$$
$$f(x_1) \upharpoonright [n,n+c+1)= f(x_\lambda)\upharpoonright [n,n+c+1).$$
Hence $f(x_0)\neq f(x_1)$. We say that $x_0$ and $x_1$ \textit{splits} $x_\lambda$ at the $1$-th interval.

At stage $s \ge 2$. For each $\sigma \in 2^s$, we can split each $x_\sigma$ for the next interval $I$ that $f(x_\sigma)\upharpoonright=10^c10^c$, since each $f(x|\sigma)$ is random. 

End construction.

\

For each real $z$, define $x_z=\lim_nx_{z\upharpoonright n}$. Then $x_{z_0} \neq x_{z_1}$ for all $z_0 \neq z_1$. Hence we get $2^{\aleph_0}$ many reals in $f(2^\omega)$. For all real $z$, $x_z$ is not normal, hence $x_z$ is non-random. 
\end{proof}

\subsection{Left-c.e.\ tests}
As pointed out in the previous sections, $f(2^\omega)$ can be constructed by steps. If we let $$f_n=\bigcup_{\sigma \in 2^{n}}[f(\sigma0^\infty),f(\sigma1^\infty)],$$
then $\lim_{n \to \infty}\mu (f_n)=0$ and each $f_n$ is a finite union of closed intervals with left-c.e. endpoints. Inspired by the non-$\emptyset'$-Kurtz-randomness of the reals in $f(2^\omega)$ and the above fact, we define two more randomness notions which are finally shown to be equivalent to 2-randomness and $\emptyset'$-Kurtz randomness respectively. 
\begin{definition}
    A left-c.e.\ test is a sequence $\{V_n=\bigcup_i[a_i^n,b_i^n]\}_{n\in\omega}$ such that $a_i^n$ and $b_i^n$ are uniformly left-c.e.\ in $i$ and $n$, $a_i^n\leq b_i^n$ and $\mu(V_n)\le2^{-n}$. A real is left-c.e.\ random if it passes all the left-c.e.\ tests.
\end{definition}

We make some remarks on the flexibility of the definition. 

Firstly, it is not necessary to require that $a_i^n\leq b_i^n$, since given any uniformly left-c.e.\ sequences $\{a_i^n\}_{i,n\in\omega}$ and $\{b_i^n\}_{i,n\in\omega}$, we can instead use the left-c.e.\ test $\{V_n=\bigcup_i[\min\{a_i^n,b_i^n\},b_i^n]\}_{n\in\omega}$, where the endpoints are also uniformly left-c.e.. 

Secondly, one can replace the closed intervals $[a_i^n,b_i^n]$ with open intervals or half-open intervals. Given a left-c.e.\ test $\{V_n=\bigcup_i[a_i^n,b_i^n]\}_{n\in\omega}$, by letting $\{U_n=\bigcup_i(a_i^n-2^{-n-i},b_i^n+2^{-n-i})\}_{n\in\omega}$, we have $\mu(U_n)\leq\mu(V_n)+2^{1-n}\leq 2^{2-n}$ so $\{U_{n+2}\}_{n\in\omega}$ is a left-c.e.\ test using open intervals and $\bigcap_n U_{n+2}\supseteq\bigcap_n V_n$. The other direction is trivial, and the same argument works for half-open intervals. 

Finally, one can require that the interior of intervals $[a_i^n,b_i^n]$ and $[a_j^n,b_j^n]$ do not overlap for any $i\neq j$. Given a left-c.e.\ test $\{V_n=\bigcup_i[a_i^n,b_i^n]\}_{n\in\omega}$ with possible overlaps, we can run the approximation of $a_i^n$ and $b_i^n$, except whenever we see an increase in $b_i^n$ that will cause $[a_i^n,b_i^n]$ and $[a_j^n,b_j^n]$ to overlap, we instead replace $[a_i^n,b_i^n]$ with $[a_i^n,a_j^n]$ and enumerate a new interval $[b_j^n,b_i^n]$. 

\begin{proposition}
    A real $x$ is left-c.e.\ random if and only if $x$ is 2-random. 
\end{proposition}
\begin{proof}
Let $\{V_n=\bigcup_{i}[a_{i}^n,b_{i}^n]\}$ be a left-c.e.\ test. Then there is a sequence $\{U_n\}$ of uniformly $\Sigma^0_1(\emptyset')$ sets such that $V_n \subseteq U_n$ and $\mu (U_n)<2^{-n+1}$. Hence every 2-random real is left-c.e. random.

On the other hand, suppose $x$ is not 2-random. Let $V_n=\bigcup_i[\sigma_i^n]$ be a $\emptyset'$-Martin-Löf test and $x\in \bigcap_nV_n$. Since $(i,n)\mapsto\sigma_{i}^n$ is computable in $\emptyset'$, by Shoenfield limit lemma there is a computable function $f:\omega^3\to 2^{<\omega}$ such that $\lim_s f(i,n,s)$ exists and equal to $\sigma_i^n$ for all $i$ and $n$. Let 
	\begin{align*}
	a_{(i,s)}^n &= \begin{cases}
		f(i,n,s)1^\omega & \text{if there is $t>s$ with $f(i,n,t)\neq f(i,n,s)$} \\
		f(i,n,s)0^\omega & \text{otherwise}
	\end{cases} ,
\end{align*}
and $b_{(i,s)}^n = f(i,n,s)1^\omega$. Then it is straightforward to verify that $a_{(i,s)}^n$ and $b_{(i,s)}^n$ are uniformly left-c.e.\ reals. Also, $\bigcup_{i,s}[a_{(i,s)}^n,b_{(i,s)}^n]=\bigcup_i[\sigma_i^n]=V_n$. Then $x$ is not left-c.e.\ random. 
\end{proof}

\begin{definition}
    A finitely left-c.e.\ test is a sequence $\{V_n=\bigcup_{i}[a_i^n,b_i^n]\}_{n\in\omega}$ such that $a_i^n$ and $b_i^n$ are uniformly left-c.e.\ in $i$ and $n$, $a_i^n\leq b_i^n$, for each $n$ there are only finitely many $i$ such that $b_i^n \neq 0$ and $\mu(V_n)\le2^{-n}$. A real is finitely left-c.e. random if it passes all the finitely left-c.e. tests.
\end{definition}

\begin{proposition}
    A real $x$ is $\emptyset'$-Kurtz random if and only if $x$ is finitely left-c.e. random. 
\end{proposition}
\begin{proof}

  Suppose that $x$ is not finitely left-c.e.\ random and let $V_n=\bigcup_{i}[a_i^n,b_i^n]$ be a finitely left-c.e.\ test capturing $x$. Given $n$, we can use $\emptyset'$ to find $k_n$ such that $b_j^n[s]=0$ for all $j\geq k_n$ and all $s$. Therefore uniformly in $n$, $\emptyset'$ computes a clopen cover $U_n$ of $V_n$ such that $\mu(U_n)\leq 2\mu(V_n)\leq 2^{1-n}$. Hence there is a $\Pi^0_1(\emptyset')$ null class $\bigcap_n U_n$ such that $x \in \bigcap_nV_n \subseteq \bigcap_n U_n$. Then $x$ is not $\emptyset'$-Kurtz random.

    On the other hand, suppose that $x$ is not $\emptyset'$-Kurtz random and let $V$ be a $\Pi^0_1(\emptyset')$ null class such that $x\in V$. Suppose $\{V_n\}$ are uniformly $\emptyset'$-computable clopen classes such that $\mu(V_n)<2^{-n}$ and $V=\bigcap_nV_n$. By Shoenfield limit lemma there is a computable function $f:\omega^2\to\text{clopen classes}$ such that $\lim_s f(n,s)=V_n$ for all $n$. 

We now construct a finitely left-c.e.\ test $\{W_n\}$ as follows. Given $n$, at stage $s+1$: If $f(n,s+1)\neq f(n,s)$, then make all intervals in $W_n$ empty by increasing their left endpoints to right endpoints. Put finitely many intervals into $W_n$ so that $W_n=f(n,s)$. 

Since $\lim_s f(n,s)$ exists, there are only finitely many stages where finitely many intervals are put into $W_n$, so $W_n$ is a union of finitely many intervals. Since $W_n=V_n$, $\{W_n\}$ is a finitely left-c.e. test capturing $x$, so $x$ is not finitely left-c.e. random.\end{proof}

Recall that $f(2^\omega)$ is a $\Pi^0_1(\emptyset')$ null class. Hence by the $\Leftarrow$ direction, it can be viewed as a finitely left-c.e. test, although we do not know the measure of the $n$-th level $f_n=\bigcup_{\sigma \in 2^{n}}[f(\sigma0^\infty),f(\sigma1^\infty)]$.

\section{Further Discussion}
\subsection{Other variants of Omega}
In this subsection we consider the relation between the function $f$ and other variants of Chaitin's Omega. Becher, Figueira, Girgorieff and Miller \cite{MR2275867} defined their variant of Omega as follows. 
\begin{definition}[\cite{MR2275867}] Given an optimal prefix-free machine $V$. For $X \subseteq 2^{<\omega}$, define
$$\Omega_V[X]=\sum_{V(\sigma)\in X}2^{-|\sigma|}.$$
\end{definition}

Recall that our function $f=f_U$ is defined by \[f_U(x)=\sum_{\sigma\leq_L x}2^{-K_U(\sigma)}.\] To put $f$ into the form of $\Omega_V[X]$, we consider \[g_V(x)=\sum_{V(\sigma)\leq_L x}2^{-|\sigma|},\] and prove that $g$ can serve as an alternative definition of $f$ in the sense of the following Fact. 
\begin{fact}\label{502}
    For each optimal prefix-free machine $U$, there is an optimal prefix-free machine $V$ such that $f_U=g_V$.
\end{fact}

\begin{proof}
  We assume that $U_s (\sigma) \downarrow$ implies $| \sigma | < s$ and at
  each stage $U$ halts on at most one new string. We enumerate a KC set as follows, which will define a prefix-free machine $V$ by the Machine Existence Theorem. 
  
  \
  
  At stage $s + 1$. For all $| \sigma | < s$:
  
  (i) If $U_s (\sigma) = \tau$, $U_{s - 1} (\sigma) \uparrow$, and there is no
  string $\sigma_1 \neq \sigma$ such that $U_s (\sigma_1) = \tau$, enumerate $(|
  \sigma |, \tau)$.
  
  (ii) If $U_s (\sigma) = \tau$, $U_{s - 1} (\sigma) \uparrow$, and there is
  a string $\sigma_1$ such that $| \sigma_1 | \leqslant | \sigma |$ and $U_s
  (\sigma_1) = \tau$, do nothing.
  
  (iii) If $U_s (\sigma) = \tau$, $U_{s - 1} (\sigma) \uparrow$, and there is
  no string $\sigma_1$ such that $| \sigma_1 | \leqslant | \sigma |$ and $U_s
  (\sigma_1) = \tau$. Find the string $\sigma_2$ of the least length such that
  $U_s (\sigma_2) = \tau$. Then $| \sigma_2 | > | \sigma |$. Enumerate $(| \sigma |
  + 1, \tau)$, $(| \sigma | + 2, \tau)$,...,$(| \sigma_2 |, \tau)$.
  
  End construction.
  
  \
  
  Then $V$ is an optimal prefix-free machine and for all $\tau$, $2^{- K
  (\tau)} = \sum_{V (\sigma) = \tau} 2^{- | \sigma |}$. Hence $f_U(x) =g_V(x)$ for all $x$.
\end{proof}

Now $f_U(x)=g_V(x)=\Omega_V[\{\sigma:\sigma\leq_L x\}]$, so $f$ can be viewed as the restriction of $\Omega_V[\cdot]$ on the sets of the form $\{\sigma:\sigma \le _L x\}$. 

\begin{theorem}[Becher et al. \cite{MR2275867}]\label{503} Let $V$ be an optimal prefix-free machine. 

    (i) (Chaitin) If $X \subseteq 2^{<\omega}$ is infinite and $\Sigma_1^0$, then $\Omega_V[X]$ is 1-random.
    
    (ii) If $y \ge \emptyset'$ and $X$ is $\Sigma_1^0(y)$-complete or $\Pi_1^0(y)$-complete, then $\Omega_V[X]$ is 1-random. In particular, if $n \ge 2$ and $X$ is $\Sigma_n^0$-complete or $\Pi_n^0$-complete, then $\Omega_V[X]$ is 1-random.

    (iii) There is a $\Delta_2^0$ set $X\subseteq 2^{<\omega}$ such that $\Omega_V[X]$ is not 1-random.
\end{theorem}

Now we can give an alternative proof of Corollary \ref{0416}: for a left-c.e.\ real $x$, the set $X=\{\sigma: \sigma\le_Lx\}$ is infinite and $\Sigma_1^0$. Then by Fact \ref{502} and Theorem \ref{503} (i), $f(x)$ is 1-random. Moreover, we have the following corollary which is not discussed in the previous sections.

\begin{corollary}
If $x$ is left-$\Sigma_n^0$-complete, left-$\Pi_n^0$-complete, right-$\Sigma_n^0$-complete or right-$\Pi_n^0$-complete for $n \ge 2$, then $f(x)$ and $\sum_{\sigma<_Lx}2^{-K(\sigma)}$ are 1-random.
\end{corollary}

\begin{remark}
    Theorem \ref{0418} and Remark \ref{remark0418} are not immediate corollaries of Theorem \ref{503} (iii) since we only consider the set $X$ of the form $\{ \sigma : \sigma \leq_L x \}$.
\end{remark}

Becher et al.\ \cite{MR2275867} also showed the following result which gives another characterization of Omega numbers. 

\begin{theorem}[Becher et al.\ \cite{MR2275867}]
If $X \subseteq 2^{<\omega}$ is c.e.\ and nonempty and $a\in[0,1]$ is left-c.e.\ and 1-random, then there is an optimal
prefix-free machine $V$ such that
\[
\Omega_V[X]= \sum_{V(\sigma)\in X} 2^{-|\sigma|}=a.
\]
\end{theorem}

The following is the similar result for our variant of Omega function. 

\begin{theorem}
  If $x\in 2^\omega$ is left-c.e.\ and $a\in[0,1]$ is left-c.e.\ and 1-random, then there is an optimal prefix-free machine $V$ such that $f_V(x)=a$.
\end{theorem}
\begin{proof}
    We first prove the theorem for $x=1^\infty$, so we need to find an optimal prefix-free machine $V$ such that $f_V(1^\infty)=a$. Since $a$ is an Omega number, there is a prefix-free optimal machine $U$ such that $\mu({\rm dom}(U))=a$. We fix a constant $c_1$ such that $(1+2^{-c_1})\mu({\rm dom}(U))<1$. We construct a prefix-free machine $V$ using the Machine Existence Theorem by enumerating a KC set. Without loss of generality, suppose that at each stage $s$, $U$ halts on at most one new string $\tau$ and gives $U(\tau)=\sigma$. We distinguish $3$ cases. \begin{enumerate}[(i)]
        \item If there is no previous $V$-description of $\sigma$, we enumerate $(|\tau|,\sigma)$. 
        \item If $|\tau|\geq|\rho|-c_1$ where $\rho$ is the previous best $V$-description of $\sigma$, we pick any $\theta$ that does not have any $V$-description so far, and enumerate $(|\tau|,\theta)$. 
        \item If $|\tau|<|\rho|-c_1$ where $\rho$ is the previous best $V$-description of $\sigma$, we enumerate $(|\tau|,\sigma)$. We also pick any $\theta$ that does not have any $V$-description so far, and enumerate $(|\rho|,\theta)$. 
    \end{enumerate}
    
    It is straightforward to verify that the total weight of the requests is at most $(1+2^{-c_1})\mu({\rm dom}(U))<1$. Also, if $\tau$ is any $U$-description of $\sigma$, then in the corresponding stage, either we're in case (i) or (iii) where $(|\tau|,\sigma)$ is enumerated, so $K_V(\sigma)\leq|\tau|$, or we're in case (ii) where there is already a $V$-description $\rho$ of $\sigma$ such that $|\rho|\leq|\tau|+c_1$, so $K_V(\sigma)\leq|\tau|+c_1$. Therefore $K_V(\sigma)\leq K_U(\sigma)+c_1$ and $V$ is optimal. Finally, by induction on stage $s$, we can verify that $f_V(1^\infty)[s]=\mu({\rm dom}(U[s]))$, therefore $f_V(1^\infty)=\mu({\rm dom}(U))=a$. This proves the Theorem for $x=1^\infty$. 

    To prove the general case, since $x$ is left-c.e., there is a computable bijection $h$ from $2^{<\omega}$ to $\{\sigma:\sigma\leq_L x\}$. Since $V$ is optimal and $h$ is computable, $K_V(\rho)\leq^+K_U(h(\rho))$. We fix a constant $c_2$ so that $K_V(\rho)\leq K_U(h(\rho))+c_2$, and that $2^{-c_2}\mu({\rm dom}(U))+\mu({\rm dom}(V))<1$. 

    We construct a prefix-free machine $W$ by the Machine Existence Theorem using the KC set containing all $(|\tau|+c_2,U(\tau))$ where $U(\tau)\downarrow$, and also all $(|\tau|,h(V(\tau)))$ where $V(\tau)\downarrow$. 
    
    It is straightforward to verify that the total weight of the requests is exactly $2^{-c_2}\mu({\rm dom}(U))+\mu({\rm dom}(V))<1$. Since each $\sigma$ with $U$-description $\tau$ is requested a $W$-description of length $|\tau|+c_2$, $K_W(\sigma)\leq K_U(\sigma)+c_2$ so $W$ is optimal. For $\sigma\leq_L x$, two types of $W$-description of $\sigma$ are requested: descriptions of length $|\tau|+c_2$ where $U(\tau)=\sigma$, the best of which has length $K_U(\sigma)+c_2$; descriptions of length $|\tau|$ where $h(V(\tau))=\sigma$, the best of which has length $K_V(h^{-1}(\sigma))\leq K_U(\sigma)+c_2$. Therefore $K_W(\sigma)=K_V(h^{-1}(\sigma))$, and \[f_W(x)=\sum_{\sigma\leq_L x}2^{-K_W(\sigma)}=\sum_{\sigma\leq_L x}2^{-K_V(h^{-1}(\sigma))}=\sum_\rho 2^{-K_V(\rho)}=f_V(1^\infty)=a.\]
\end{proof}

\begin{corollary} 
For a real $a \in (0,1)$, the following are equivalent.

(i) $a$ is left-c.e.\ and 1-random.

(ii) There is an optimal prefix-free machine $V$ and a left-c.e.\ real $x$ such that $f_V(x)=0.a$.

(iii) For all left c.e.\ real $x$, there is an optimal prefix-free machine $V$ such that $f_V(x)=0.a$. 
\end{corollary}

 Although we do not pursue this idea further here, it is worth noting that $f$ can be extended to a binary function 
$$F(x,y):=\sum_{\sigma \le _L y}2^{-K^x(\sigma)}.$$
$F$ could also be worth studying since it is a generalization of our $f$ and the Omega operator $x\mapsto \sum_{U^x(\sigma)\downarrow} 2^{-|\sigma|}$ defined in \cite{downey2005relativizing}, where $U$ is a universal prefix-free oracle machine. Although we do not known whether the range of the Omega operator in \cite{downey2005relativizing} must be Borel, the range of $F$ is a perfect set as $f(2^\omega)$. And a real $y$ is weakly low for $K^x$ if and only if $F(x,y)$ is $x \oplus y$-random. Hence the measure of  the range $rng(F)$ is positive since it contains a $z$-random real for all $z \in 2^\omega$ (see \cite{downey2005relativizing} Theorem 5.3 for more details). Also $\mu(rng(F))<1$ since $\sup_{x \in 2^\omega}\sum_{U^x(\sigma)\downarrow} 2^{-|\sigma|}<1$ (see \cite{downey2005relativizing} section 9).

\subsection{Questions}
One may ask whether the Corollary \ref{comp} above is vacuous. In other words, is there a computable real in $f(2^\omega)$?
\begin{question}
    Is there a computable real in $f(2^\omega)$? Given a computable real (or just a rational) $p$, is there is an optimal prefix-free machine $V$ such that $p\in f_V(2^\omega)$?
\end{question}

We already know that $f$ is Turing invariant in the ideal induced by the $K$-trivial reals but it is not Turing invariant in general.
\begin{question}
    If $x$ is not $K$-trivial, can $f$ be Turing invariant on $deg(x)$?
\end{question}

\begin{question}\label{q2}
    Is there $x$ such that $f(x)$ is of hyperimmune-free degree?
\end{question}

If the answer of Question \ref{q2} is affirmative, fix a real $x$ such that $f(x)$ is of hyperimmune-free degree, then by a result of Yu (see \cite{downey2010algorithmic} Theorem 8.11.12), $f(x)$ is Kurtz random if and only if it is weakly 2-random. Since none of the reals in $f(2^\omega)$ is weakly 2-random, $f(x)$ is not Kurtz random. Also by Lemma 3.9 of Ng et al. \cite{MR3289547}, $f(2^\omega)$ as a $\Pi_1^{0}(\emptyset')$ set, has a $\Pi_1^0$ subclass containing $f(x)$. Then by the Jochusch-Soare low basis theorem (see \cite{MR882921} VI.5), there is a real $y$ such that $f(y)$ is low.

\section*{Acknowledgements}
Part of this work is contained in the second author’s master thesis under the guidance of Liang Yu. The authors would like to thank Prof.\ Yu for the helpful discussions. The first author would like to thank Theodore A. Slaman and Ruizhi Yang for their guidance and time.

\end{document}